\documentclass[11pt, final]{article}
\usepackage{a4}
\usepackage{amsmath}%
\usepackage{amstext}%
\usepackage{amssymb}%
\usepackage{showkeys}%
\usepackage{epsfig}%
\usepackage{cite}
\usepackage{tikz}
\usepackage{subfig}
\usepackage{caption}
\usepackage{multirow}
\usepackage{booktabs}
\usepackage{floatrow}
\allowdisplaybreaks

\usepackage{listings}
\newtheorem{theorem}{Theorem}

\newtheorem{axiom}{Axiom}

\newtheorem{definition}[axiom]{Definition}

\newtheorem{lemma}[theorem]{Lemma}

\newtheorem{pb}[theorem]{Problem}
\newenvironment{problem}{\pb\rm}{\endpb}

\newtheorem{proposition}[theorem]{Proposition}
\newtheorem{rem}[theorem]{Remark}
\newenvironment{remark}{\rem\rm}{\endrem}

\newcounter{unnumber}

\newtheorem{prf}[unnumber]{Proof.}
\newenvironment{proof}{\prf\rm}{\hfill{$\blacksquare$}\endprf}
\newcommand{\R}{\mathbb{R}}%
\newcommand{\ol}{\overline}%
\DeclareMathOperator*\dom{dom}%
\DeclareMathOperator*\B{\overline{\R}}%
\DeclareMathOperator*\gr{Gr}%
\DeclareMathOperator*\ran{ran}%
\DeclareMathOperator*\id{Id}%
\DeclareMathOperator*\argmin{argmin}

\DeclareMathOperator*\zer{zer}

\title{Approaching the solving of constrained variational inequalities via penalty term-based dynamical systems}

\author{Radu Ioan Bo\c{t} \thanks{University of Vienna, Faculty of Mathematics, Oskar-Morgenstern-Platz 1, A-1090 Vienna, Austria,
email: radu.bot@univie.ac.at. Research partially supported by DFG (German Research Foundation), project BO 2516/4-1.} \and
Ern\"{o} Robert Csetnek \thanks {University of Vienna, Faculty of Mathematics, Oskar-Morgenstern-Platz 1, A-1090 Vienna, Austria,
email: ernoe.robert.csetnek@univie.ac.at. Research supported by FWF (Austrian Science Fund), Lise Meitner Programme, project M 1682-N25.}}

\begin{document}
\maketitle

\noindent \textbf{Abstract.} We investigate the existence and uniqueness of (locally) absolutely continuous trajectories of a penalty term-based 
dynamical system associated to a constrained variational inequality expressed as a monotone inclusion problem. Relying on Lyapunov analysis and on the
ergodic continuous version of the celebrated Opial Lemma we prove weak ergodic convergence of the orbits 
to a solution of the constrained variational inequality under investigation. If one of the operators involved satisfies stronger monotonicity properties, then 
strong convergence of the trajectories can be shown.\vspace{1ex}

\noindent \textbf{Key Words.} dynamical systems, Lyapunov analysis, monotone inclusions, forward-backward algorithm, penalty schemes \vspace{1ex}

\noindent \textbf{AMS subject classification.} 34G25, 47J25, 47H05, 90C25

\section{Introduction and preliminaries}\label{sec-intr}

This paper is motivated by the increasing interest in solving constrained variational inequalities expressed as monotone inclusion problems of the form
\begin{equation}\label{att-cza-peyp-p}
0\in Ax+N_C(x),
\end{equation}
where ${\cal H}$ is a real Hilbert space, $A:{\cal H}\rightrightarrows{\cal H}$ is a maximally monotone operator, $C=\argmin \Psi$ is the set 
of global minima of the  proper, convex and lower semicontinuous function $\Psi : {\cal H} \rightarrow \overline{\R}:=\R\cup\{\pm\infty\}$ 
fulfilling $\min \Psi=0$ and $N_C:{\cal H}\rightrightarrows{\cal H}$ is the normal cone of the set $C \subseteq {\cal H}$ 
(see \cite{att-cza-10, att-cza-peyp-c, att-cza-peyp-p, noun-peyp, peyp-12, b-c-penalty-svva, b-c-penalty-vjm, banert-bot-pen}). 
One can find in the literature iterative schemes based on the forward-backward paradigm for solving \eqref{att-cza-peyp-p} 
(see \cite{att-cza-peyp-c, att-cza-peyp-p, noun-peyp, peyp-12}), 
that perform in each iteration a proximal step with respect to $A$ and a subgradient step with respect to the penalization function $\Psi$.

Recently, even more complex structures have been analyzed, like monotone inclusion problems of the form
\begin{equation}\label{att-cza-peyp-p-gen}
0\in Ax+Dx+N_C(x),
\end{equation}
where $A :{\cal H}\rightrightarrows{\cal H}$ is a maximally monotone operator, $D :{\cal H}\rightarrow {\cal H}$ is a (single-valued) 
cocoercive operator and $C \subseteq {\cal H}$ is the (nonempty) set of zeros of another cocoercive operator $B :{\cal H}\rightarrow {\cal H}$, see \cite{b-c-penalty-svva, b-c-penalty-vjm, banert-bot-pen}. 

In this paper we are concerned with addressing monotone inclusion problem \eqref{att-cza-peyp-p-gen} from the perspective of dynamical systems. 
More precisely, we associate to this constrained variational inequality a first-order dynamical system  formulated in terms of the resolvent of the 
maximal monotone operator $A$, which has as discrete counterparts penalty-type numerical schemes already considered in the literature in the context
of solving \eqref{att-cza-peyp-p-gen}. Let us mention that dynamical systems of similar implicit type have been investigated in \cite{antipin, bolte-2003, abbas-att-arx14, b-c-dyn-KM}.

In the first part of the manuscript we study the existence and uniqueness of (locally) absolutely continuous trajectories generated by the 
dynamical system, by appealing to arguments based on the Cauchy-Lipschitz-Picard Theorem (see \cite{haraux, sontag}). In the second part 
of the paper we investigate the convergence of the trajectories to a solution of the constrained variational inequality \eqref{att-cza-peyp-p-gen}. 
We use as tools Lyapunov analysis combined with the continuous version of the Opial Lemma. Under the fulfillment of a condition expressed 
in terms of the Fitzpatrick function of the cocoercive operator $B$ we are able to show ergodic weak convergence of the orbits. 
Moreover, if the operator $A$ is strongly monotone, we can prove even strong (non-ergodic) convergence for the generated trajectories. 

For the reader's convenience we present in the following some notations which are used throughout the paper 
(see \cite{bo-van, bauschke-book, simons}).

Let ${\cal H}$ 
be a real Hilbert space with \textit{inner product} $\langle\cdot,\cdot\rangle$ and associated \textit{norm} 
$\|\cdot\|=\sqrt{\langle \cdot,\cdot\rangle}$. The \textit{normal cone} of $S\subseteq {\cal H}$ is defined by $N_S(x)=\{u\in {\cal H}:\langle y-x,u\rangle\leq 0 \ \forall y\in S\}$, if $x\in S$ and $N_S(x)=\emptyset$ 
for $x\notin S$. Notice that for $x\in S$, $u\in N_S(x)$ if and only if $\sigma_S(u)=\langle x,u\rangle$, where $\sigma_S$ is the 
support function of $S$, defined by $\sigma_S(u)=\sup_{y\in S}\langle y,u\rangle$.

For an arbitrary set-valued operator $M:{\cal H}\rightrightarrows {\cal H}$ we denote by 
$\gr M=\{(x,u)\in {\cal H}\times {\cal H}:u\in Mx\}$ its \emph{graph}, by $\dom M=\{x \in {\cal H} : Mx \neq \emptyset\}$ its 
\emph{domain}, by $\ran M=\{u\in {\cal H}: \exists x\in {\cal H} \mbox{ s.t. }u\in Mx\}$ its \emph{range} and 
$M^{-1}:{\cal H}\rightrightarrows {\cal H}$ its \emph{inverse operator}, defined by $(u,x)\in\gr M^{-1}$ if and only if $(x,u)\in\gr M$. 
We use also the notation $\zer M=\{x\in {\cal H}: 0\in Mx\}$ for the \textit{set of zeros} of the operator $M$. 
We say that $M$ is \emph{monotone} if $\langle x-y,u-v\rangle\geq 0$ for all $(x,u),(y,v)\in\gr M$. A monotone operator $M$ is 
said to be \emph{maximally monotone}, if there exists no proper monotone extension of the graph of $M$ on ${\cal H}\times {\cal H}$. 
Let us mention that in case $M$ is maximally monotone, $\zer M$ is a convex and closed set \cite[Proposition 23.39]{bauschke-book}. 
We refer to \cite[Section 23.4]{bauschke-book} for 
conditions ensuring that $\zer M$ is nonempty.  If $M$ is maximally monotone, then one has the following characterization for the set of its zeros:
\begin{equation}\label{charact-zeros-max}
z\in\zer M \mbox{ if and only if }\langle u-z,w\rangle\geq 0\mbox{ for all }(u,w)\in \gr M.
\end{equation}

The operator $M$ is said to be \emph{$\gamma$-strongly monotone} with $\gamma>0$, if $\langle x-y,u-v\rangle\geq \gamma\|x-y\|^2$ for 
all $(x,u),(y,v)\in\gr M$. Notice that if $M$ is maximally monotone and strongly monotone, then $\zer M$ is a singleton, 
thus nonempty (see \cite[Corollary 23.37]{bauschke-book}).

The \emph{resolvent} of $M$, $J_M:{\cal H} \rightrightarrows {\cal H}$, is defined by $J_M=(\id+M)^{-1}$, where 
$\id :{\cal H} \rightarrow {\cal H}, \id(x) = x$ for all $x \in {\cal H}$, is the \textit{identity operator} on ${\cal H}$. 
Moreover, if $M$ is maximally monotone, then $J_M:{\cal H} \rightarrow {\cal H}$ is single-valued and maximally monotone 
(cf. \cite[Proposition 23.7 and Corollary 23.10]{bauschke-book}). We will also use the \emph{Yosida approximation} of the operator $M$, 
which is defined by $M_{\alpha}=\frac{1}{\alpha}(\id -J_{\alpha M})$, for $\alpha>0$. 

The \emph{Fitzpatrick function} associated to a monotone operator $M$, defined as
$$\varphi_M:{\cal H}\times {\cal H}\rightarrow \B, \ \varphi_M(x,u)=\sup_{(y,v)\in\gr M}\{\langle x,v\rangle+\langle y,u\rangle-\langle y,v\rangle\},$$
is a convex and lower semicontinuous function and it will play an important role throughout the paper. Introduced by Fitzpatrick in \cite{fitz}, 
this notion opened the gate towards the employment of convex analysis specific tools when investigating the maximality of monotone operators 
(see \cite{bauschke-book, bausch-m-s, bo-van, bu-sv-02, simons, b-hab, borw-06, BCW-set-val} and the references therein). 
In case $M$ is maximally monotone, $\varphi_M$ is proper and it fulfills
$$\varphi_M(x,u)\geq \langle x,u\rangle \ \forall (x,u)\in {\cal H}\times {\cal H},$$
with equality if and only if $(x,u)\in\gr M$. We refer the reader to \cite{bausch-m-s}, for formulae of the corresponding Fitzpatrick functions 
computed for particular classes of monotone operators.

Let $\gamma>0$ be arbitrary. A single-valued operator $M:{\cal H}\rightarrow {\cal H}$ is said to be \textit{$\gamma$-cocoercive}, if 
$\langle x-y,Mx-My\rangle\geq \gamma\|Mx-My\|^2$ for all $(x,y)\in {\cal H}\times {\cal H}$, and \textit{$\gamma$-Lipschitz continuous}, 
if $\|Mx-My\|\leq \gamma\|x-y\|$ for all $(x,y)\in {\cal H}\times {\cal H}$. 

In this paper we are concerned with the solving of the following constrained variational inequality expressed as monotone inclusion problem 
(see also \cite{b-c-penalty-svva}). 

\begin{problem}\label{pr-cocoercive-single-val}
Let ${\cal H}$ be a real Hilbert space, $A:{\cal H}\rightrightarrows {\cal H}$ a maximally monotone operator, 
$D:{\cal H}\rightarrow{\cal H}$ an $\eta$-cocoercive operator with $\eta>0$, $B:{\cal H}\rightarrow{\cal H}$ a 
$\mu$-cocoercive operator with $\mu>0$ and suppose that $C=\zer B \neq\emptyset$. The monotone inclusion problem to solve is
$$0\in Ax+Dx+N_C(x).$$
\end{problem}

Let us mention that a (discrete) iterative scheme for solving this problem has been proposed and investigated in \cite{att-cza-peyp-c} 
for $D$ taken as zero operator and $B$ as the gradient of a convex and differentiable function with Lipschitz continuous gradient. 

\section{A penalty term-based dynamical system}\label{sc2}

We associate to Problem \ref{pr-cocoercive-single-val} the following dynamical system: 

\begin{equation}\label{dyn-syst-pen-fb}\left\{
\begin{array}{ll}
\dot x(t)+x(t)=J_{\lambda(t) A}\Big(x(t)-\lambda(t) D x(t)-\lambda(t)\beta(t) Bx(t)\Big)\\
x(0)=x_0,
\end{array}\right.\end{equation}
where $x_0\in {\cal H}$ is fixed and $\lambda,\beta:[0,+\infty)\rightarrow(0,+\infty)$. 

\begin{remark} (i) The dynamical system \eqref{dyn-syst-pen-fb} can be seen as an extension of similar implicit first-order constructions considered in the last years in the literature.
For instance, the resulting dynamical system when $B$ is the zero operator and $\lambda$ is a constant function has been investigated 
in \cite{b-c-dyn-KM} in connection with approaching the set of zeros of $A+D$. Moreover, the situation when $A$ is the convex subdifferential of a proper, convex and 
lower semicontinuous function has been addressed in \cite{abbas-att-arx14}, while the even more particular case when this function 
is the indicator function of a nonempty, convex and closed subset of ${\cal H}$ has been considered in \cite{bolte-2003}.

(ii) The explicit discretization of \eqref{dyn-syst-pen-fb} with respect to the time variable $t$, with step size $h_n>0$, yields for 
an initial point $x_0\in {\cal H}$ the following iterative scheme 
$$\frac{x_{n+1}-x_n}{h_n}+x_n=J_{\lambda_n A}\Big(x_n-\lambda_n D x_n-\lambda_n\beta_n Bx_n\Big) \ \forall n \geq 0,$$
which for $h_n=1$ becomes 
\begin{equation}\label{discrete}x_{n+1}=J_{\lambda_n A}\Big(x_n-\lambda_n D x_n-\lambda_n\beta_n Bx_n\Big) \ \forall n \geq 0,\end{equation}
where $(\lambda_n)_{n\geq 0},(\beta_n)_{n\geq 0}$ are sequences of positive real numbers. 

Let us mention that a convergence analysis for \eqref{discrete} has been carried out in \cite{b-c-penalty-svva}. The case when $D$ is the zero operator
has been addressed in \cite{att-cza-peyp-c} under the supplementary assumption that $B$ is the gradient of a convex and differentiable function with 
Lipschitz continuous gradient. Other penalty-type iterative schemes have been considered in the context of solving monotone inclusion problems and 
convex optimization problems in \cite{att-cza-peyp-p, peyp-12, noun-peyp, banert-bot-pen, b-c-penalty-vjm}. 
\end{remark}

As in \cite{att-sv2011, abbas-att-sv}, we consider the following definition of an absolutely continuous function.

\begin{definition}\label{abs-cont} \rm A function $f:[0,b]\rightarrow {\cal H}$ (where $b>0$) is said to be absolutely continuous if one of the 
following equivalent properties holds: 

(i)  there exists an integrable function $g:[0,b]\rightarrow {\cal H}$ such that $$f(t)=f(0)+\int_0^t g(s)ds \ \ \forall t\in[0,b];$$

(ii) $f$ is continuous and its distributional derivative is Lebesgue integrable on $[0,b]$; 

(iii) for every $\varepsilon > 0$, there exists $\eta >0$ such that for any finite family of intervals $I_k=(a_k,b_k)$ we have the implication:
$$\left(I_k\cap I_j=\emptyset \mbox{ and }\sum_k|b_k-a_k| < \eta\right)\Longrightarrow \sum_k\|f(b_k)-f(a_k)\| < \varepsilon.$$
\end{definition}

\begin{remark}\label{rem-abs-cont}\rm (a) It follows from the above definition that an absolutely continuous function is differentiable almost 
everywhere, its derivative coincides with its distributional derivative almost everywhere and one can recover the function from its derivative $f'=g$ 
by the integration formula (i). 

(b) If $f:[0,b]\rightarrow {\cal H}$ (where $b>0$) is absolutely continuous and $B:{\cal H}\rightarrow {\cal H}$ is $L$-Lipschitz continuous
(where $L\geq 0$), then the function $h=B\circ f$ is absolutely continuous, too. This can be easily verified by considering the characterization in
Definition \ref{abs-cont}(iii). Moreover, $h$ is almost everywhere differentiable and the inequality $\|h'(\cdot)\|\leq L\|f'(\cdot)\|$ holds almost everywhere.   
\end{remark}

\begin{definition}\label{str-sol}\rm We say that $x:[0,+\infty)\rightarrow {\cal H}$ is a strong global solution of \eqref{dyn-syst-pen-fb} if the 
following properties are satisfied: 

(i) $x:[0,+\infty)\rightarrow {\cal H}$ is absolutely continuous on each interval $[0,b]$, $0<b<+\infty$; 

(ii) $\dot x(t)+x(t)=J_{\lambda(t) A}(x(t)-\lambda(t) D x(t)-\lambda(t)\beta(t) Bx(t))$ for almost every $t\in[0,+\infty)$;

(iii) $x(0)=x_0$.
\end{definition}

In what follows we discuss the existence and uniqueness of strong global solutions of \eqref{dyn-syst-pen-fb}. To this end we use the Cauchy-Lipschitz theorem for absolutely continuous trajectories (see for example 
\cite[Proposition 6.2.1]{haraux}, \cite[Theorem 54]{sontag}). To this end we will make use of the following Lipschitz property of the resolvent operator as a function of the step size, which actually is a consequence of 
the classical results \cite[Proposition 2.6]{brezis} and \cite[Proposition 23.28]{bauschke-book}; see also \cite[Proposition 3.1]{abbas-att-sv}. 

\begin{proposition}\label{prop-Lipsch-res-step-size} Let $A:{\cal H}\rightrightarrows {\cal H}$ be a maximally monotone operator, $x\in {\cal H}$ and $0<\delta<+\infty$. 
Then the mapping $\tau\mapsto J_{\tau A}x$ is Lipschitz continuous on $[\delta,+\infty)$. More precisely, for any $\lambda,\mu\in[\delta,+\infty)$ the following inequality holds:
\begin{equation}\label{Lipsch-res-step-size} \|J_{\lambda A}x-J_{\mu A}x\|\leq |\lambda-\mu|\|A_\delta x\|.
\end{equation}
\end{proposition}

For proving the existence of strong global solutions of \eqref{dyn-syst-pen-fb}, we need the following natural assumption:
$${\rm (H1)} \ \  \lambda,\beta:[0,+\infty)\rightarrow(0,+\infty) \mbox{ are continuous on each interval }[0,b], \ \mbox{for} \ 0<b<+\infty.$$

Notice that the dynamical system \eqref{dyn-syst-pen-fb} can be written as 
\begin{equation}\label{existence}\left\{
\begin{array}{ll}
\dot x(t)=f(t,x(t))\\
x(0)=x_0,
\end{array}\right.\end{equation}
where $f:[0,+\infty)\times {\cal H}\rightarrow {\cal H}$ is defined by 
\begin{equation}\label{f-existence}f(t,x)=\left[J_{\lambda(t)A}\circ \big(\id - \lambda(t)D-\lambda(t)\beta(t)B\big)-\id\right]x.\end{equation}

(a) We claim that for every $t\geq 0$ and every $x,y\in {\cal H}$ we have  

\begin{equation}\label{lipsch-x}\|f(t,x)-f(t,y)\|\leq \left(2+\frac{\lambda(t)}{\eta}+\frac{\lambda(t)\beta(t)}{\mu}\right)\|x-y\|.\end{equation}

Indeed, since the resolvent operator is nonexpansive (see \cite[Corollary 23.10 and Definition 4.1]{bauschke-book}), $D$ is ($1/\eta$)-Lipschitz continuous
and $B$ is ($1/\mu$)-Lipschitz continuous, it holds
\begin{align*}
\|f(t,x)-f(t,y)\|  \leq & \left\|J_{\lambda(t) A}\big(x-\lambda(t)Dx-\lambda(t)\beta(t)Bx\big)-
J_{\lambda(t) A}\big(y-\lambda(t)Dy-\lambda(t)\beta(t)By\big)\right\|\\
                      & +\|x-y\|\\
                     \leq & \lambda(t)\|Dx-Dy\|+\lambda(t)\beta(t)\|Bx-By\|+2\|x-y\|\\
                     \leq & \left(2+\frac{\lambda(t)}{\eta}+\frac{\lambda(t)\beta(t)}{\mu}\right)\|x-y\|,
\end{align*}
hence \eqref{lipsch-x} holds. Further, notice that due to (H1), 
$$L_f: [0,+\infty) \rightarrow \R, L_f(t) = 2+\frac{\lambda(t)}{\eta}+\frac{\lambda(t)\beta(t)}{\mu},$$
which is for every $t \geq 0$ equal to the Lipschitz-constant of $f(t,\cdot)$, satisfies
$$L_f(\cdot)\in L^1([0,b]) \mbox{ for any } 0<b<+\infty.$$

(b) We show now that \begin{equation}\label{existence-b}\forall x\in{\cal H}, \ \forall b>0, \ \ f(\cdot,x)\in L^1([0,b],{\cal H}).\end{equation}

Let us fix $x\in{\cal H}$ and $b>0$. Due to (H1), there exist $\lambda_{min},\beta_{min}>0$ such that
$$0<\lambda_{min}\leq\lambda(t)\mbox{ and }0<\beta_{min}\leq\beta(t) \ \forall t\in[0,b].$$

We obtain for all $t\in[0,b]$ the following chain of inequalities: 
\begin{align*}
\|f(t,x)\| \leq & \ \|x\|+\left\|J_{\lambda(t)A}\big(x - \lambda(t)Dx-\lambda(t)\beta(t)Bx\big)\right\|\\
\leq & \ \|x\|+\|x-\lambda(t)Dx-\lambda(t)\beta(t)Bx-x+\lambda_{min}Dx+\lambda_{min}\beta_{min}Bx\|\\
& \ +\left\|J_{\lambda(t)A}\big(x - \lambda_{min}Dx-\lambda_{min}\beta_{min}Bx\big)\right\|\\
\leq & \ \|x\|+(\lambda(t)-\lambda_{min})\|Dx\|+(\lambda(t)\beta(t)-\lambda_{min}\beta_{min})\|Bx\|\\
& \ +(\lambda(t)-\lambda_{min})\|A_{\lambda_{min}}(x-\lambda_{min}Dx-\lambda_{min}\beta_{min}Bx)\|\\
& \ +\left\|J_{\lambda_{min}A}\big(x - \lambda_{min}Dx-\lambda_{min}\beta_{min}Bx\big)\right\|,
\end{align*}
where in the second inequality we used the nonexpansiveness of the resolvent operator and in the third one the statement of Proposition \ref{prop-Lipsch-res-step-size}. 
The claim \eqref{existence-b} follows now easily by integrating and by taking into account (H1). 

In the light of the statements proven in (a) and (b), the existence and uniqueness of a strong global solution of the dynamical system 
\eqref{dyn-syst-pen-fb} follow from \cite[Proposition 6.2.1]{haraux} (see also \cite[Theorem 54]{sontag}).

\section{Convergence of the generated trajectories}

In this section we investigate the convergence properties of the trajectories generated by the dynamical system \eqref{dyn-syst-pen-fb}. Our analysis 
relies on Lyapunov analysis combined with the continuous ergodic version of the Opial Lemma.  

We split the proof of the convergence into several lemmas. 

\begin{lemma}\label{l-fej1} 
Consider the setting of Problem \ref{pr-cocoercive-single-val} and the associated dynamical system \eqref{dyn-syst-pen-fb} under the assumption 
that {\rm (H1)} holds. Take $(z,w)\in\gr (A+D+N_C)$ such that $w=v+p+Dz$, where $v\in Az$ and $p\in N_C(z)$. 
Then the following inequality holds for almost every $t\geq 0$
\begin{align}\label{fej1}
& \frac{d}{dt}\|x(t)-z\|^2+\lambda(t)(2\eta-3\lambda(t))\|Dx(t)-Dz\|^2\leq \nonumber \\ 
& 2\lambda(t)\beta(t)\left[\sup_{u\in C}\varphi_B\left(u,\frac{p}{\beta(t)}\right)-\sigma_C\left(\frac{p}{\beta(t)}\right)\right] +3\lambda^2(t)\beta^2(t)\|Bx(t)\|^2 + 3\lambda^2(t)\|Dz+v\|^2+\nonumber\\
& 2\lambda(t)\langle z-x(t),w\rangle.
\end{align}
\end{lemma}

\begin{proof} From the definition of the resolvent we have for almost every $t \geq 0$ 
$$-\frac{1}{\lambda(t)}\dot x(t)-Dx(t)-\beta(t)Bx(t)\in A(\dot x(t)+x(t)),$$ 
which combined with $v\in Az$ and the monotonicity of $A$ gives 
$$\langle \dot x(t)+x(t)-z, \dot x(t)\rangle\leq \lambda(t)\langle z-\dot x(t)-x(t), \beta(t)Bx(t)+Dx(t)+v\rangle.$$
From here it follows that for almost every $t \geq 0$
\begin{align*}
\frac{d}{dt}\|x(t)-z\|^2 = & \ 2\langle \dot x(t)+x(t)-z, \dot x(t)\rangle - 2\|\dot x(t)\|^2 \\
\leq & \ 2\lambda(t)\langle z-\dot x(t)-x(t), \beta(t)Bx(t)+Dx(t)+v\rangle-2\|\dot x(t)\|^2 \\
= & \ 2\lambda(t)\langle z-x(t), \beta(t)Bx(t)+Dx(t)+v\rangle\\
& \ +2\lambda(t)\langle -\dot x(t), \beta(t)Bx(t)+Dx(t)+v\rangle-2\|\dot x(t)\|^2\\
\leq & \ 2\lambda(t)\langle z-x(t), \beta(t)Bx(t)+Dx(t)+v\rangle+\lambda^2(t)\|\beta(t)Bx(t)+Dx(t)+v\|^2\\
\leq & \ 2\lambda(t)\langle z-x(t), \beta(t)Bx(t)+Dx(t)+v\rangle+3\lambda^2(t)\beta^2(t)\|Bx(t)\|^2\\
& \ +3\lambda^2(t)\|Dz+v\|^2+3\lambda^2(t)\|Dx(t)-Dz\|^2.
\end{align*}
It remains to evaluate the first term on the right-hand side of the last of the above inequalities. By noticing that $v=w-Dz-p$, from the 
cocoercivity of $D$, the definition of the Fitzpatrick function and using that $\sigma_C\left(\frac{p}{\beta(t)}\right)=\left\langle z,\frac{p}{\beta(t)}\right\rangle$, we obtain for almost every $t \geq 0$
\begin{align*}
& 2\lambda(t)\langle z-x(t), \beta(t)Bx(t)+Dx(t)+w-Dz-p\rangle= \\
& 2\lambda(t)\langle z-x(t), Dx(t)-Dz\rangle+2\lambda(t)\langle z-x(t), \beta(t)Bx(t)-p\rangle+2\lambda(t)\langle z-x(t),w\rangle = \\
& 2\lambda(t)\langle z-x(t), Dx(t)-Dz\rangle +\\ 
& 2\lambda(t)\beta(t)\left[\langle z,Bx(t)\rangle+\left\langle x(t),\frac{p}{\beta(t)}\right\rangle- \langle x(t),Bx(t)\rangle-\left\langle z,\frac{p}{\beta(t)}\right\rangle\right] + 2\lambda(t)\langle z-x(t),w\rangle \leq\\
& -2\eta\lambda(t)\|Dx(t)-Dz\|^2+2\lambda(t)\beta(t)\left[\sup_{u\in C}\varphi_B\left(u,\frac{p}{\beta(t)}\right)-\sigma_C\left(\frac{p}{\beta(t)}\right)\right] + 2\lambda(t)\langle z-x(t),w\rangle
\end{align*}
and the desired conclusion follows.  
\end{proof}

\begin{lemma}\label{l-fej2} Consider the setting of Problem \ref{pr-cocoercive-single-val} and the associated dynamical system \eqref{dyn-syst-pen-fb} under the assumption 
that {\rm (H1)} holds. Let  $z\in C\cap\dom A$ and $v\in Az$. Then for every $\varepsilon\geq 0$ and almost every $t\geq 0$ we have
\begin{align}\label{fej2}
& \frac{d}{dt}\|x(t)-z\|^2+ \frac{1+2\varepsilon}{1+\varepsilon}\|\dot x(t)\|^2+\frac{2\varepsilon}{1+\varepsilon}\lambda(t)\beta(t)\langle x(t)-z,Bx(t)\rangle \leq \nonumber \\
& \lambda(t)\beta(t)\left((1+\varepsilon)\lambda(t)\beta(t)-\frac{2\mu}{1+\varepsilon}\right)\|Bx(t)\|^2+2\lambda(t)\langle z-\dot x(t)-x(t),Dx(t)+v\rangle.
\end{align}
\end{lemma}

\begin{proof} As it follows from the first inequality obtained in the proof of Lemma \ref{l-fej1}, we have for almost every $t \geq 0$
\begin{align*}
 \frac{d}{dt}\|x(t)-z\|^2 + 2\|\dot x(t)\|^2 \leq & \ 2\lambda(t)\beta(t)\langle z-x(t),Bx(t)\rangle+2\lambda(t)\beta(t)\langle -\dot x(t),Bx(t)\rangle\\
                                                & \ +2\lambda(t)\langle z-\dot x(t)-x(t),Dx(t)+v\rangle.
\end{align*}
Since $B$ is $\mu$-cocoercive and $Bz=0$, we have for almost every $t \geq 0$
$$\langle z-x(t),Bx(t)\rangle\leq -\mu\|Bx(t)\|^2,$$
hence $$2\lambda(t)\beta(t)\langle z-x(t),Bx(t)\rangle\leq -\frac{2\mu}{1+\varepsilon}\lambda(t)\beta(t)\|Bx(t)\|^2+ \frac{2\varepsilon}{1+\varepsilon}\lambda(t)\beta(t)\langle z-x(t),Bx(t)\rangle.$$

The conclusion follows by using that for almost every $t \geq 0$
$$2\lambda(t)\beta(t)\langle -\dot x(t),Bx(t)\rangle\leq \frac{1}{1+\varepsilon}\|\dot x(t)\|^2+(1+\varepsilon)\lambda^2(t)\beta^2(t)\|Bx(t)\|^2.$$
\end{proof}

\begin{lemma}\label{l-fej3} Consider the setting of Problem \ref{pr-cocoercive-single-val} and the associated dynamical system \eqref{dyn-syst-pen-fb} under the assumption 
that {\rm (H1)} holds. Moreover, suppose that $\limsup_{t\rightarrow+\infty}\lambda(t)\beta(t)<2\mu$ and let be $z\in C\cap\dom A$ and $v\in Az$. Then there exist $a,b>0$ and $t_0 >0$ such that 
for almost every $t\geq t_0$ it holds
\begin{align}\label{fej3}
& \frac{d}{dt}\|x(t)-z\|^2+a\left(\|\dot x(t)\|^2+\lambda(t)\beta(t)\langle x(t)-z,Bx(t)\rangle + \lambda(t)\beta(t)\|Bx(t)\|^2\right) \leq \nonumber \\
& \left(b\lambda^2(t)-2\eta\lambda(t)\right)\|Dx(t)-Dz\|^2+2\lambda(t)\langle z-x(t),v+Dz\rangle+b\lambda^2(t)\|Dz+v\|^2.
\end{align}
\end{lemma}

\begin{proof} Let us evaluate the second term of right-hand side of the inequality \eqref{fej2}. The cococercivity of $D$ yields for almost every $t \geq 0$
\begin{align*}
& 2\lambda(t)\langle z-\dot x(t)-x(t),Dx(t)+v\rangle=\\
& 2\lambda(t)\langle -\dot x(t),Dx(t)+v\rangle+2\lambda(t)\langle z-x(t),Dx(t)-Dz\rangle +2\lambda(t)\langle z-x(t),v+Dz\rangle \leq \\
& \frac{\varepsilon}{2(1+\varepsilon)}\|\dot x(t)\|^2+\frac{2(1+\varepsilon)}{\varepsilon}\lambda^2(t)\|Dx(t)+v\|^2 -\\
& 2\eta\lambda(t)\|Dx(t)-Dz\|^2+2\lambda(t)\langle z-x(t),v+Dz\rangle \leq\\
& \frac{\varepsilon}{2(1+\varepsilon)}\|\dot x(t)\|^2+\frac{4(1+\varepsilon)}{\varepsilon}\lambda^2(t)\|Dx(t)-Dz\|^2+ \frac{4(1+\varepsilon)}{\varepsilon}\lambda^2(t)\|Dz+v\|^2 - \\
& 2\eta\lambda(t)\|Dx(t)-Dz\|^2+2\lambda(t)\langle z-x(t),v+Dz\rangle.
\end{align*}
Combining this inequality with \eqref{fej2} we obtain for almost every $t \geq 0$
\begin{align*}
& \frac{d}{dt}\|x(t)-z\|^2+\frac{2+3\varepsilon}{2(1+\varepsilon)}\|\dot x(t)\|^2+ \frac{2\varepsilon}{1+\varepsilon}\lambda(t)\beta(t)\langle x(t)-z,Bx(t)\rangle+\frac{\varepsilon}{1+\varepsilon}\lambda(t)\beta(t)\|Bx(t)\|^2 \leq\\
& \lambda(t)\beta(t)\left((1+\varepsilon)\lambda(t)\beta(t)-\frac{2\mu}{1+\varepsilon}+\frac{\varepsilon}{1+\varepsilon}\right)\|Bx(t)\|^2 +\\ 
& \left(\frac{4(1+\varepsilon)}{\varepsilon}\lambda^2(t)-2\eta\lambda(t)\right)\|Dx(t)-Dz\|^2 +\\
& 2\lambda(t)\langle z-x(t), Dz+v\rangle+\frac{4(1+\varepsilon)}{\varepsilon}\lambda^2(t)\|Dz+v\|^2.
\end{align*}
Further, there exist $\alpha$ and $\varepsilon_0>0$ such that $$\limsup_{t\rightarrow+\infty}\lambda(t)\beta(t)<\alpha<2\mu$$ and 
$$(1+\varepsilon_0)\alpha-\frac{2\mu}{1+\varepsilon_0}+\frac{\varepsilon_0}{1+\varepsilon_0}<0.$$ By taking 
$a:=\frac{\varepsilon_0}{2(1+\varepsilon_0)}$ and $b:=\frac{4(1+\varepsilon_0)}{\varepsilon_0}$ the desired conclusion follows.
\end{proof}

\begin{lemma}\label{l-fej4} Consider the setting of Problem \ref{pr-cocoercive-single-val} and the associated dynamical system \eqref{dyn-syst-pen-fb} under the assumption 
that {\rm (H1)} holds. Moreover, suppose that $\limsup_{t\rightarrow+\infty}\lambda(t)\beta(t)<2\mu$ and $\lim\inf_{t\rightarrow+\infty}\lambda(t)=0$ and let be $(z,w)$ $\in\gr(A+D+N_C)$ such that $w=v+p+Dz$, 
where $v\in Az$ and $p\in N_C(z)$. Then there exist $a,b >0$ and $t_1>0$ such that for almost every $t\geq t_1$ it holds
\begin{align}\label{fej4}
& \frac{d}{dt}\|x(t)-z\|^2+a\left(\|\dot x(t)\|^2+\frac{\lambda(t)\beta(t)}{2}\langle x(t)-z,Bx(t)\rangle + \lambda(t)\beta(t)\|Bx(t)\|^2\right) \leq \nonumber\\
& \frac{a\lambda(t)\beta(t)}{2}\left[\sup_{u\in C}\varphi_B\left(u,\frac{4p}{a\beta(t)}\right)-\sigma_C\left(\frac{4p}{a\beta(t)}\right)\right]+
2\lambda(t)\langle z-x(t),w\rangle+b\lambda^2(t)\|Dz+v\|^2.
\end{align}
\end{lemma}

\begin{proof}
According to Lemma \ref{l-fej3}, there exist $a,b >0$ and $t_0 >0$ such that for almost every  $t \geq t_0$ the inequality \eqref{fej3} holds. 
Since $\lim\inf_{t\rightarrow\infty}\lambda(t)=0$, there exists $t_1 \geq t_0$ such that $b\lambda^2(t)-2\eta\lambda(t)\leq 0$ for every $t\geq t_1$, 
hence for almost every $t \geq t_1$
\begin{align*}
& \frac{d}{dt}\|x(t)-z\|^2+a\left(\|\dot x(t)\|^2+\lambda(t)\beta(t)\langle x(t)-z,Bx(t)\rangle + \lambda(t)\beta(t)\|Bx(t)\|^2\right) \leq \\
&  2\lambda(t)\langle z-x(t),v+Dz\rangle+b\lambda^2(t)\|Du+v\|^2.
\end{align*}
The conclusion follows by combining this inequality with the following one, which holds for almost every $t \geq 0$
\begin{align*}
& 2\lambda(t)\langle z-x(t),v+Dz\rangle+\frac{a\lambda(t)\beta(t)}{2}\langle z-x(t),Bx(t)\rangle = \\
& 2\lambda(t)\langle z-x(t),-p\rangle+\frac{a\lambda(t)\beta(t)}{2}\langle z-x(t),Bx(t)\rangle+2\lambda(t)\langle z-x(t),w\rangle = \\
& \frac{a\lambda(t)\beta(t)}{2}\left(\langle z,Bx(t)\rangle +\left\langle x(t),\frac{4p}{a\beta(t)}\right\rangle-\langle x(t),Bx(t)\rangle-
\left\langle z,\frac{4p}{a\beta(t)}\right\rangle\right)+2\lambda(t)\langle z-x(t),w\rangle \leq \\
& \frac{a\lambda(t)\beta(t)}{2}\left[\sup_{u\in C}\varphi_B\left(u,\frac{4p}{a\beta(t)}\right)-
\sigma_C\left(\frac{4p}{a\beta(t)}\right)\right]+2\lambda(t)\langle z-x(t),w\rangle.
\end{align*}
\end{proof}
For proving the convergence statement for the trajectories generated by the dynamical system \eqref{dyn-syst-pen-fb} we will make use of the following ergodic version of the continuous Opial Lemma. 
The proof of this results follows similarly to the one of \cite[Lemma 2.3]{att-cza-10} and therefore we omit it. 

\begin{lemma}\label{opial} Let $S \subseteq {\cal H}$ be a nonempty set, $x:[0,+\infty)\rightarrow{\cal H}$ a given map and 
$\lambda:[0,+\infty)\rightarrow(0,+\infty)$ such that $\int_0^{+\infty}\lambda(t)=+\infty$. Define $\tilde x :[0, +\infty)\rightarrow {\cal H}$ by 
$$\tilde x(t)=\frac{1}{\int_0^t\lambda(s)ds}\int_0^t\lambda(s)x(s)ds.$$
Assume that 

(i) for every $z\in S$, $\lim_{t\rightarrow+\infty}\|x(t)-z\|$ exists; 

(ii) every weak sequential cluster point of the map $\tilde x$ belongs to $S$. 

\noindent Then there exists $x_{\infty}\in S$ such that $w-\lim_{t\rightarrow+\infty}\tilde x(t)=x_{\infty}$. 
\end{lemma}

We will prove the convergence results under the following hypotheses, which can be seen as continuous counterparts of the conditions considered in 
\cite{b-c-penalty-svva} in the discrete case (see also \cite{att-cza-peyp-c, att-cza-10}): 
\begin{enumerate}
\item[{\rm (H2)}] $A+N_C$ is maximally monotone and $\zer(A+D+N_C)\neq\emptyset$;

\item[{\rm (H3)}] $\lambda(\cdot)\in L^2([0,+\infty))\setminus L^1([0,+\infty))$; 

\item[$(H_{fitz})$] For every $p\in\ran N_C$, $\int_0^{+\infty} 
\lambda(t)\beta(t)\left[\sup\limits_{u \in C}\varphi_B\left(u,\frac{p}{\beta(t)}\right)-\sigma_C\left(\frac{p}{\beta(t)}\right)\right]dt<+\infty$.
\end{enumerate}

\begin{remark}\label{h} 
\begin{enumerate}
\item[(a)] Since $A$ is maximally monotone and $C$ is a nonempty, convex and closed set, $A + N_C$ is maximally monotone, provided that a so-called regularity condition is fulfilled. We refer the reader to
\cite{bauschke-book, b-hab, BCW-set-val, bo-van, borw-06, simons, Zal-carte} for conditions guaranteeing the maximal monotonicity of the 
sum of two maximally monotone operators. Further, as $D$ is maximally monotone (see \cite[Example 20.28]{bauschke-book}) and $\dom D={\cal H}$, {\rm (H2)} guarantees 
that $A+D+N_C$ is maximally monotone, too (see \cite[Corollary 24.4]{bauschke-book}). 

\item[(b)] The condition (H3) is obviously satisfied for the function $\lambda(t)=\frac{1}{t+1}$. 

\item[(c)] Let us turn now our attention to $(H_{fitz})$. The discrete version of this condition has been considered for the first time in \cite{b-c-penalty-svva}. We notice that for each $p\in\ran N_C$ we have
$$\sup\limits_{u \in C}\varphi_B\left(u,\frac{p}{\beta(t)}\right)-\sigma_C\left(\frac{p}{\beta(t)}\right)\geq 0 \ \forall t\geq 0.$$
Indeed, if $p\in\ran N_C$, then there exists $\ol u \in C$ such that $p\in N_C(\ol u)$. This implies that
$$\sup\limits_{u \in C}\varphi_B\left(u,\frac{p}{\beta(t)}\right)-\sigma_C\left(\frac{p}{\beta(t)}\right)
\geq \left\langle \ol u,\frac{p}{\beta_n}\right\rangle-\sigma_C\left(\frac{p}{\beta_n}\right)=0 \ \forall t\geq 0.$$
Let us consider the particular case $B=\nabla \Psi$, where $\Psi:{\cal H}\rightarrow\R$ is a convex and differentiable function with Lipschitz 
continuous gradient and satisfies $\min\Psi=0$. In this case $C=\argmin\Psi$, $\Psi(x)=0$ for $x\in C$ and it holds (see \cite{bausch-m-s})
\begin{equation}\label{fitzp-subdiff-ineq}
 \varphi_{\nabla\Psi }(x,u)\leq \Psi(x) +\Psi^*(u) \ \forall (x,u)\in {\cal H}\times {\cal H},
\end{equation}
where $\Psi^*:{\cal H}\rightarrow\R\cup\{+\infty\}$, $\Psi^*(u)=\sup_{x\in {\cal H}}\{\langle u,x\rangle-\Psi(x)\}$, is the Fenchel conjugate of $\Psi$. 

This means that $(H_{fitz})$ is in this particular case fulfilled, if one has: 

$(H)$ For every $p\in\ran N_C$, $\int_0^{+\infty} \lambda(t)\beta(t)\left[\Psi^*\left(\frac{p}{\beta(t)}\right)-\sigma_C\left(\frac{p}{\beta(t)}\right)\right]dt<+\infty$. 

Let us mention that $(H)$ is the continuous counterpart of a condition used in \cite{att-cza-peyp-c} in the context of proving convergence for penalty-type iterative schemes. 
It has its origins in the work \cite{att-cza-10}, where a similar condition has been used in the convergence analysis of a coupled dynamical systems with multiscale aspects. 

Let us present a particular setting in which $(H)$ and, consequently, $(H_{fitz})$ are fulfilled. This example is inspired 
by \cite[Section 1.3(b)]{att-cza-10}. Take $\Psi:{\cal H}\rightarrow\R$, 
$\Psi(x)=\frac{1}{2}d^2(x,C)=\frac{1}{2}\inf_{y\in C}\|x-y\|^2$. For its conjugate function one gets $\Psi^*(x)=\frac{1}{2}\|x\|^2+\sigma_C(x) \ \forall x \in {\cal H}$, 
hence $(H)$ reduces to $$\int_0^{+\infty}\frac{\lambda(t)}{\beta(t)}dt<+\infty,$$
which is obviously fulfilled for $\lambda(t)=\frac{1}{t+1}$ and $\beta(t)=1+t$. For other particular instances where $(H_{fitz})$ (in its continuous or discrete version) holds we refer the reader to 
\cite{att-cza-10, att-cza-peyp-c, att-cza-peyp-p, banert-bot-pen, peyp-12, noun-peyp}. 
\end{enumerate}
\end{remark}

Let us state now the main result concerning the convergence of the trajectories generated by the dynamical system \eqref{dyn-syst-pen-fb}. 

\begin{theorem}\label{conv-dyn-pen} 
Consider the setting of Problem \ref{pr-cocoercive-single-val} and the associated dynamical system \eqref{dyn-syst-pen-fb}. Assume that 
$\limsup_{t\rightarrow+\infty}\lambda(t)\beta(t)<2\mu$ and that {\rm (H1)-(H3)} and $(H_{fitz})$ hold. 
Let $\tilde x:[0,+\infty)\rightarrow {\cal H}$ be defined by $$\tilde x(t)=\frac{1}{\int_0^t\lambda(s)ds}\int_0^t\lambda(s)x(s)ds.$$
Then the following statements are true:
\begin{itemize}
\item[(i)] for every $z\in\zer(A+D+N_C)$, $\|x(t)-z\|$ converges as $t\rightarrow+\infty$; moreover, $\int_0^{+\infty}\|\dot x(t)\|^2 dt$ $<+\infty$, $\int_0^{+\infty}\lambda(t)\beta(t)\langle Bx(t),x(t)-z\rangle dt \!<+\infty$ 
and $\!\int_0^{+\infty}\lambda(t)\beta(t)\|B x(t)\|^2 dt$ $< +\infty$; 

\item[(ii)] $\tilde x(t)$ converges weakly to an element in $\zer(A+D+N_C)$ as $t\rightarrow+\infty$;

\item[(iii)] if, additionally, $A$ is strongly monotone, then $x(t)$ converges strongly to the unique element in 
$\zer(A+D+N_C)$ as $t\rightarrow+\infty$.
\end{itemize}
\end{theorem}

\begin{proof} (i) According to {\rm (H3)}, the function $\lambda(\cdot)$ satisfies the relation $\lim\inf_{t\rightarrow\infty}\lambda(t)=0$. Take $z\in\zer(A+D+N_C)$ and
$v\in Az$ and $p\in N_C(z)$ fulfilling $0=v+p+Dz$. Applying Lemma \ref{l-fej4} for $w=0$, it follows that there exist $a,b >0$ and $t_1>0$ such that for almost every $t\geq t_1$ it holds
\begin{align*}
& \frac{d}{dt}\|x(t)-z\|^2+a\left(\|\dot x(t)\|^2+\frac{\lambda(t)\beta(t)}{2}\langle x(t)-z,Bx(t)\rangle + \lambda(t)\beta(t)\|Bx(t)\|^2\right) \leq \\
& \frac{a\lambda(t)\beta(t)}{2}\left[\sup_{u\in C}\varphi_B\left(u,\frac{4p}{a\beta(t)}\right)-\sigma_C\left(\frac{4p}{a\beta(t)}\right)\right]+ b\lambda^2(t)\|Dz+v\|^2.
\end{align*}
Since the function (having as argument $t$) on the right-hand side of the above inequality belongs to $L^1([0,+\infty))$, by using also \cite[Lemma 5.1]{abbas-att-sv}, the statements follow.

(ii) According to Lemma \ref{opial}, it is enough to show that every weak sequential cluster limit of $\tilde x$ belongs to $\zer(A+D+N_C)$. Let 
$\ol x$ be such a weak sequential cluster limit, that is, there exists a sequence $(t_n)_{n \geq 0} \rightarrow+\infty$ such that $\tilde x(t_n)$ weakly converges to 
$\ol x$ as $n\rightarrow+\infty$. 

Take an arbitrary $(z,w)\in\gr(A+D+N_C)$ such that $w=v+p+Dz$, where $v\in Az$ and $p\in N_C(z)$. From Lemma \ref{l-fej1} and by using that $\lim\inf_{t\rightarrow+\infty}\lambda(t)=0$ it follows that there exists 
$t_2>0$ such that for almost every $t\geq t_2$ we have 
\begin{align*}
\frac{d}{dt}\|x(t)-z\|^2 \leq & \ 2\lambda(t)\beta(t)\left[\sup_{u\in C}\varphi_B\left(u,\frac{p}{\beta(t)}\right)-\sigma_C\left(\frac{p}{\beta(t)}\right)\right]\nonumber \\
& \ +3\lambda^2(t)\beta^2(t)\|Bx(t)\|^2+3\lambda^2(t)\|Dz+v\|^2+2\lambda(t)\langle z-x(t),w\rangle.
\end{align*}
By integrating from $t_2$ to $T$, where $T\geq t_2$, we obtain 
\begin{equation}\label{conv-ineq1}\|x(T)-z\|^2-\|x(t_2)-z\|^2\leq L+
2\left\langle\left(\int_{t_2}^T\lambda(t)dt\right)z-\int_{t_2}^T\lambda(t)x(t)dt,w\right\rangle,\end{equation}
where, according to the hypotheses and statement (i), 
\begin{align*}
L:= & \ 2\int_0^{+\infty} \lambda(t)\beta(t)\left[\sup\limits_{u \in C}\varphi_B\left(u,\frac{p}{\beta(t)}\right)-\sigma_C\left(\frac{p}{\beta(t)}\right)\right]dt\\
& \ +3\int_0^{+\infty}\lambda^2(t)\beta^2(t)\|Bx(t)\|^2dt+3\|Dz+v\|^2\int_0^{+\infty}\lambda^2(t)dt<+\infty.
\end{align*}
Now dividing \eqref{conv-ineq1} by $\int_0^T\lambda(t)dt$ and discarding the nonnegative term $\|x(T)-z\|^2$, we obtain 
\begin{equation}\label{conv-ineq2}\frac{-\|x(t_2)-z\|^2}{\int_0^T\lambda(t)dt}\leq \frac{L'}{\int_0^T\lambda(t)dt}+
2\left\langle z-\frac{1}{\int_0^T\lambda(t)dt}\int_0^T\lambda(t)x(t)dt, w\right\rangle, \end{equation}
where 
$$L':=L+2\left\langle\left(-\int_0^{t_2}\lambda(t)dt\right)z+\int_0^{t_2}\lambda(t)x(t)dt,w\right\rangle<+\infty.$$
Letting $T:=t_n$ in \eqref{conv-ineq2} for any $n \geq 0$, passing to $n\rightarrow+\infty$ and using (H3) and the definition of $\tilde x$, it follows
$$0\leq 2\langle z-\ol x, w\rangle.$$ 
Since  $(z,w)$ was taken arbitrary in $\gr(A+D+N_C)$, we obtain from \eqref{charact-zeros-max}  that $\ol x\in\zer(A+D+N_C)$ and from here the conclusion follows. 

(iii) Suppose that $A$ is $\gamma$-strongly monotone, where $\gamma>0$. Let $z$ be the unique element in $\zer(A+D+N_C)$ and $v\in Az$ and $p\in N_C(z)$ such that $0=v+p+Dz$.  Following the lines of the proof of Lemma \ref{l-fej1}, one 
can prove that for almost every $t \geq 0$
\begin{align*}\label{fej1-str1}
& 2\gamma\lambda(t)\|\dot x(t)+x(t)-z\|^2+\frac{d}{dt}\|x(t)-z\|^2+\lambda(t)(2\eta-3\lambda(t))\|Dx(t)-Dz\|^2 \leq \nonumber\\
& 2\lambda(t)\beta(t)\left[\sup_{u\in C}\varphi_B\left(u,\frac{p}{\beta(t)}\right)-\sigma_C\left(\frac{p}{\beta(t)}\right)\right] +3\lambda^2(t)\beta^2(t)\|Bx(t)\|^2+3\lambda^2(t)\|Dz+v\|^2.
\end{align*}

Since $\liminf_{t\rightarrow+\infty}\lambda(t)=0$, it follows that there exists $t_2>0$ such that for almost every $t\geq t_2$
\begin{align*}
& 2\gamma\lambda(t)\|\dot x(t)+x(t)-z\|^2+\frac{d}{dt}\|x(t)-z\|^2 \leq \\
& 2\lambda(t)\beta(t)\left[\sup_{u\in C}\varphi_B\left(u,\frac{p}{\beta(t)}\right)-\sigma_C\left(\frac{p}{\beta(t)}\right)\right] +3\lambda^2(t)\beta^2(t)\|Bx(t)\|^2+3\lambda^2(t)\|Dz+v\|^2,
\end{align*}
thus
\begin{align*}
& \gamma\lambda(t)\|x(t)-z\|^2+\frac{d}{dt}\|x(t)-z\|^2 \leq 2\gamma\lambda(t)\|\dot x(t)\|^2+\\ & 
2\lambda(t)\beta(t)\left[\sup_{u\in C}\varphi_B\left(u,\frac{p}{\beta(t)}\right)-\sigma_C\left(\frac{p}{\beta(t)}\right)\right]
 +3\lambda^2(t)\beta^2(t)\|Bx(t)\|^2+3\lambda^2(t)\|Dz+v\|^2.
\end{align*}
By using the hypotheses and statement (i), after integration of the last inequality one obtains 
$$\int_0^{+\infty}\lambda(t)\|x(t)-z\|^2dt<+\infty.$$
Using the convergence of $\|x(t)-z\|$ as $t\rightarrow+\infty$ and (H3) it follows that $x(t)$ must converge to $z$ as $t\rightarrow+\infty$. 
\end{proof}

\end{document}